\documentclass[12pt]{amsart}

\usepackage{amsmath}
\usepackage{amssymb}
\usepackage{fancybox}
\usepackage{ascmac}
\usepackage{amsthm}
\usepackage{latexsym}
\usepackage[dvips]{graphicx}
\usepackage{graphics}

\newtheorem{theorem}{Theorem}[section]
\newtheorem{lemma}[theorem]{Lemma}
\newtheorem{corollary}[theorem]{Corollary}
\newtheorem{proposition}[theorem]{Proposition}

\newtheorem{example}[theorem]{Example}

\newcommand*{\thmlab}[1]{\label{thm:#1}}
\newcommand*{\thmref}[1]{Theorem~\ref{thm:#1}}
\newcommand*{\corlab}[1]{\label{cor:#1}}

\newcommand*{\proplab}[1]{\label{prop:#1}}
\newcommand*{\propref}[1]{Proposition~\ref{prop:#1}}

\newcommand*{\lemlab}[1]{\label{lem:#1}}
\newcommand*{\lemref}[1]{Lemma~\ref{lem:#1}}
\newcommand*{\exlab}[1]{\label{ex:#1}}
\newcommand*{\exref}[1]{Example~\ref{ex:#1}}
\def\zz{{\mathbb Z}}

\def\calg{{\mathcal G}}

\def\bfa{{\bf a}}

\begin{document}

\title{Strong Koszulness of the toric ring associated to a cut ideal}
\author{Kazuki Shibata}

\thanks{
{\bf 2010 Mathematics Subject Classification:}
Primary 05E40.
\\
\, \, \, 
{\bf Keywords:}
cut ideal, Gr\"obner basis, strongly Koszul algebra.
\\
The author is supported by JSPS Research Fellowship for Young Scientists.
}

\address{Kazuki Shibata,
Department of Mathematics,
Graduate School of Science,
Rikkyo University,
Toshima-ku, Tokyo 171-8501, Japan.}
\email{12rc003c@rikkyo.ac.jp}

\date{}

\begin{abstract}
A cut ideal of a graph was introduced by Sturmfels and Sullivant.
In this paper, we give a necessary and sufficient condition for  toric rings associated to the cut ideal to be strongly Koszul.
\end{abstract}

\maketitle

\section*{Introduction}

Let $G$ be a finite simple graph on the vertex set $V(G)=[n]=\{ 1,\ldots , n\}$ with the edge set $E(G)$.
For two subsets $A$ and $B$ of $[n]$ such that $A \cap B = \emptyset$ and $A \cup B = [n]$, the $(0,1)$-vector $\delta_{A|B} (G) \in \zz^{|E(G)|}$ is defined as
$$
\delta_{A|B} (G)_{ij} =
\begin{cases}
1 & \text{if~$ |A \cap \{ i , j \} | =1$,} \\
0 & \text{otherwise,}
\end{cases}
$$
where $ij$ is an edge of $G$.
Let
$$
X_G = \left\{
\binom{\delta_{A_1 | B_1}(G)}{1} , \ldots , 
\binom{\delta_{A_N | B_N}(G)}{1}
\right\}
\subset \zz^{|E(G)|+1}
\quad (N=2^{n-1}).
$$
As necessary, we consider $X_G$ as the collection of vectors or as the matrix.
Let $K$ be a field and 
\begin{eqnarray*}
K[q] &=& K[q_{A_1|B_1} , \ldots , q_{A_N|B_N}], \\
K[s,T] &=& K[s,t_{ij}~|~  ij \in E(G)]
\end{eqnarray*}
be two polynomial rings over $K$.
Then the ring homomorphism is defined as follows:
$$
\pi_G ~:~ K[q] \rightarrow K[s,T],
\quad
q_{A_l|B_l} \mapsto s \cdot \prod_{\substack{|A_l \cap \{ i,j \}| =1 \\ ij \in E(G)}} t_{ij}
$$
for $1 \le l \le N$.
The {\it cut ideal} $I_G$ of $G$ is the kernel of $\pi_G$ and the {\it toric ring} $R_G$ of $X_G$ is the image of $\pi_G$.
We put $u_{A|B} = \pi_G (q_{A|B})$.

In \cite{StuSull}, Sturmfels and Sullivant introduced a cut ideal and posed the problem of relating properties of cut ideals to the class of graphs.

Let $R$ be a semigroup ring and $I$ be the defining ideal of $R$.
We say that $R$ is {\it compressed} if the initial ideal of $I$ is squarefree with respect to any reverse lexicographic order.
For the toric ring $R_G$ and the cut ideal $I_G$, the following results are known:

\begin{theorem}
[\cite{StuSull}]
The toric ring $R_G$ is compressed if and only if $G$ has no $K_5$-minor and every induced cycle in $G$ has length $3$ or $4$.
\end{theorem}

\begin{theorem}
[\cite{Engstrom}]
The cut ideal $I_G$ is generated by quadratic binomials if and only if $G$ has no $K_4$-minor.
\end{theorem}

Nagel and Petrovi\'c showed that the cut ideal $I_G$ associated with ring graphs has a quadratic Gr\"obner basis \cite{NaPe}.
However we do not know generally when the cut ideal $I_G$ has a quadratic Gr\"obner basis and when $R_G$ is Koszul except for trivial cases.

On the other hand, the notion of strongly Koszul algebras was introduced by Herzog, Hibi and Restuccia \cite{HHR}.
A strongly Koszul algebra is a stronger notion of Koszulness.
In general, it is known that, for a semigroup ring $R$,
$$
\begin{array}{c}
\text{The~defining~ideal~of~$R$~has~a~quadratic~Gr\"obner~basis,~or~$R$~is~strongly~Koszul} \\
\Downarrow \\
\text{$R$~is~Koszul} \\
\Downarrow \\
\text{The~defining~ideal~of~$R$~is~generated~by~quadratic~binomials.}
\end{array}
$$
We do not know whether the defining ideal of a strongly Koszul semigroup ring has a quadratic Gr\"obner basis.
In \cite{ReRi}, Restuccia and Rinaldo gave a sufficient condition for toric rings to be strongly Koszul.
In \cite{MatsudaOhsugi}, Matsuda and Ohsugi proved that any squarefree strongly Koszul toric ring is compressed.

In this paper, we give a sufficient condition for cut ideals to have a quadratic Gr\"obner basis and we characterize the class of graphs such that $R_G$ is strongly Koszul.

The outline of this paper is as follows.
In Section 1, we show that the set of graphs such that $R_G$ is strongly Koszul is closed under contracting edges, induced subgraphs and $0$-sums.
In Section 2, we compute Gr\"obner basis for the cut ideal without $(K_4,C_5)$-minor.
In Section 3, by using results of Section 1 and Section 2, we prove that the toric ring $R_G$ is strongly Koszul if and only if $G$ has no $(K_4,C_5)$-minor.

\section{Clique sums and strongly Koszul algebras}

In this paper, we introduce the equivalent condition as the definition of the strongly Koszul algebra.

Let $R$ be a semigroup ring generated by $u_1, \ldots , u_n$.
We say that a semigroup ring $R$ is {\it strongly Koszul} if the ideals $(u_i) \cap (u_j)$ are generated in degree $2$ for all $i \ne j$ \cite[Proposition 1.4]{HHR}.
\begin{proposition}[{\cite[Proposition 2.3]{HHR}}]
\proplab{tensor}
Let $R$ and $P$ be semigroup rings over same field, and $Q$ be the tensor product or the Segre product of $R$ and $P$.
Then $Q$ is strongly Koszul if and only if both $R$ and $P$ are strongly Koszul.
\end{proposition}

Recall that a graph $H$ is a {\it minor} of a graph $G$ if $H$ can be obtained by deleting and contracting edges of $G$.
We say that a subgraph $H$ is an {\it induced subgraph} of a graph $G$ if $H$ contains all the edges $ij \in E(G)$ with $i,j \in V(H)$.
\begin{proposition}
\proplab{contraction}
Let $G$ be a finite simple connected graph.
Assume that $R_G$ is strongly Koszul.
Then
\begin{itemize}
\item[(1)] If $H_1$ is an induced subgraph of $G$, then $R_{H_1}$ is strongly Koszul.
\item[(2)] If $H_2$ is obtained by contracting an edge of $G$, then $R_{H_2}$ is strongly Koszul.
\end{itemize}
\end{proposition}

\begin{proof}
By \cite{Ohsugi} and \cite{StuSull},
$R_{H_1}$ and $R_{H_2}$ are combinatorial pure subrings of $R_G$.
Therefore, by \cite[Corollary 1.6]{OHH}, $R_{H_1}$ and $R_{H_2}$ are strongly Koszul.
\end{proof}


Let $G_1=(V_1 , E_1)$ and $G_2=(V_2,E_2)$ be simple graphs such that $V_1 \cap V_2$ is a clique of both graphs.
The new graph $G= G_1 \# G_2$ with the vertex set $V_1 \cup V_2$ and the edge set $E_1 \cup E_2$ is called the {\it clique sum} of $G_1$ and $G_2$ along $V_1 \cap V_2$.
If the cardinality of $V_1 \cap V_2$ is $k+1$,
then this operation is called a $k$-{\it sum} of the graphs.
It is clear that if $R_{G_1 \# G_2}$ is strongly Koszul, then both $R_{G_1}$ and $R_{G_2}$ are strongly Koszul because $G_1$ and $G_2$ are induced subgraphs of $G_1 \# G_2$.

\begin{proposition}
\proplab{0-sum}
The set of graphs $G$ such that $R_G$ is strongly Koszul is closed under the $0$-sum.
\end{proposition}

\begin{proof}
Let $G_1$ and $G_2$ be finite simple connected graphs and assume that $R_{G_1}$ and $R_{G_2}$ are strongly Koszul.
Then the toric ring $R_{G_1 \# G_2}$, where $G_1 \# G_2$ is the $0$-sum of $G_1$ and $G_2$,
is the usual Segre product of $R_{G_1}$ and $R_{G_2}$.
Thus it follows by \propref{tensor}.
\end{proof}

However the set of graphs $G$ such that $R_G$ is strongly Kosuzl is not always closed under the $1$-sum.

Let $K_{n}$ denote the complete graph on $n$ vertices, $C_n$ denote the cycle of length $n$ and $K_{l_1,\ldots,l_m}$ denote the complete $m$-partite graph on the vertex set $V_{1} \cup \cdots \cup V_{m}$, where $|V_i| = l_i$ for $1 \le i \le m$ and $V_{i} \cap V_{j} = \emptyset$ for $i \ne j$.

\begin{example}
\exlab{exam}
{\rm Let $G_1=C_3 \# C_3(=K_4 \setminus e)$, $G_2=C_4 \# C_3$ and $G_3 =(K_4 \setminus e) \# C_3$ be graphs shown in Figures 1-3.}
{\rm All of $R_{C_3}$, $R_{C_4}$ and $R_{G_1}$ are strongly Koszul because $R_{C_3}$ is isomorphic to the polynomial ring and $I_{C_4}$ and $I_{G_1}$ have quadratic Gr\"obner bases with respect to any reverse lexicographic order, respectively (see \cite{ReRi,StuSull}).}
{\rm However neither $R_{G_2}$ nor $R_{G_3}$ is strongly Koszul since $(u_{\emptyset|[5]}) \cap (u_{\{ 1,3,4 \} | \{ 2,5 \}})$ is not generated in degree 2.}

\begin{center}
\begin{figure}[htb]
 \begin{minipage}{0.3\hsize}
  \begin{center}
\includegraphics[trim = 90 100 90 0 , scale=.09, clip]{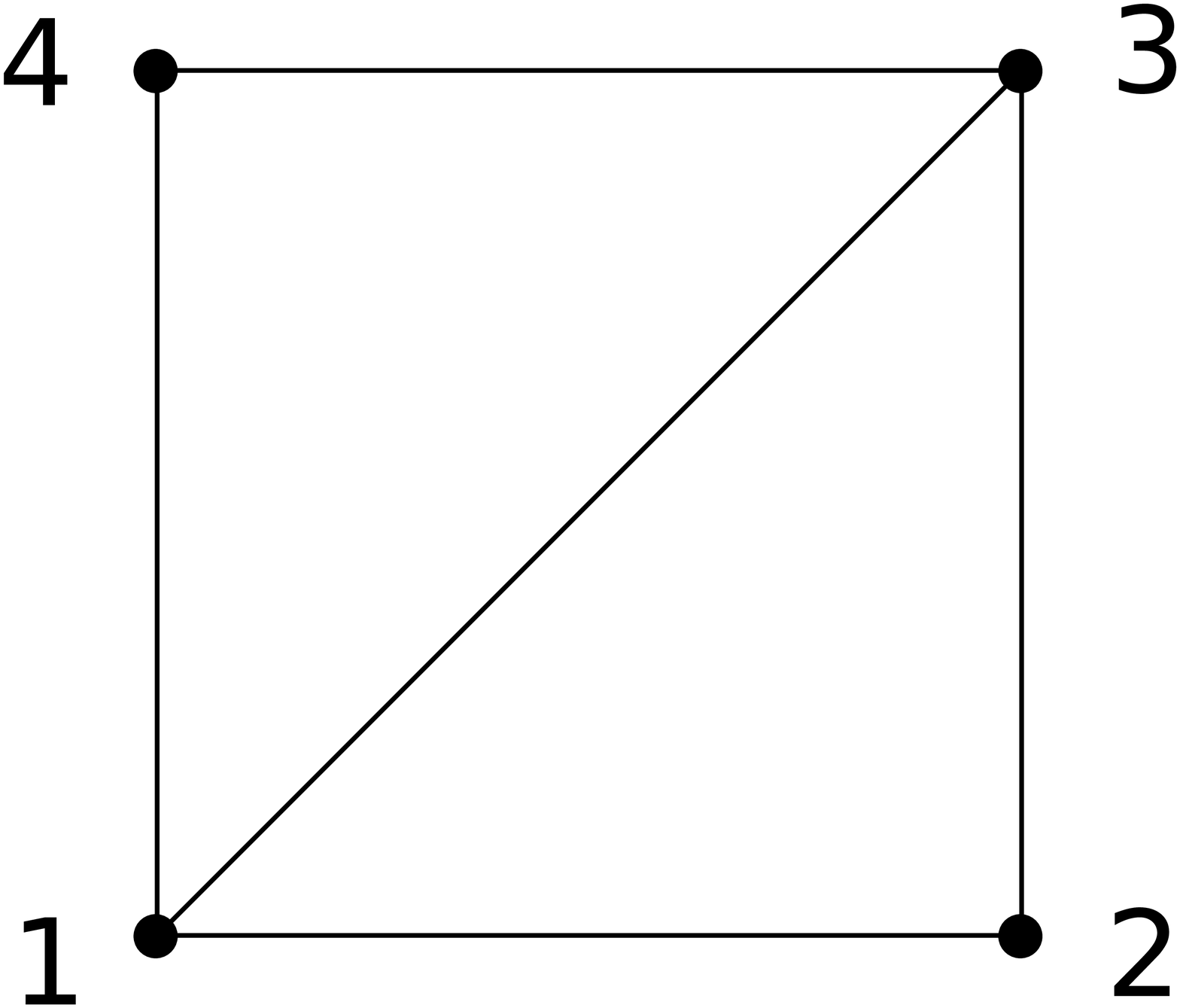}
  \end{center}
  \caption{$C_3 \# C_3$}
  \label{fig:one}
 \end{minipage}
 \begin{minipage}{0.3\hsize}
 \begin{center}
\includegraphics[trim = 90 90 90 0 , scale=.09, clip]{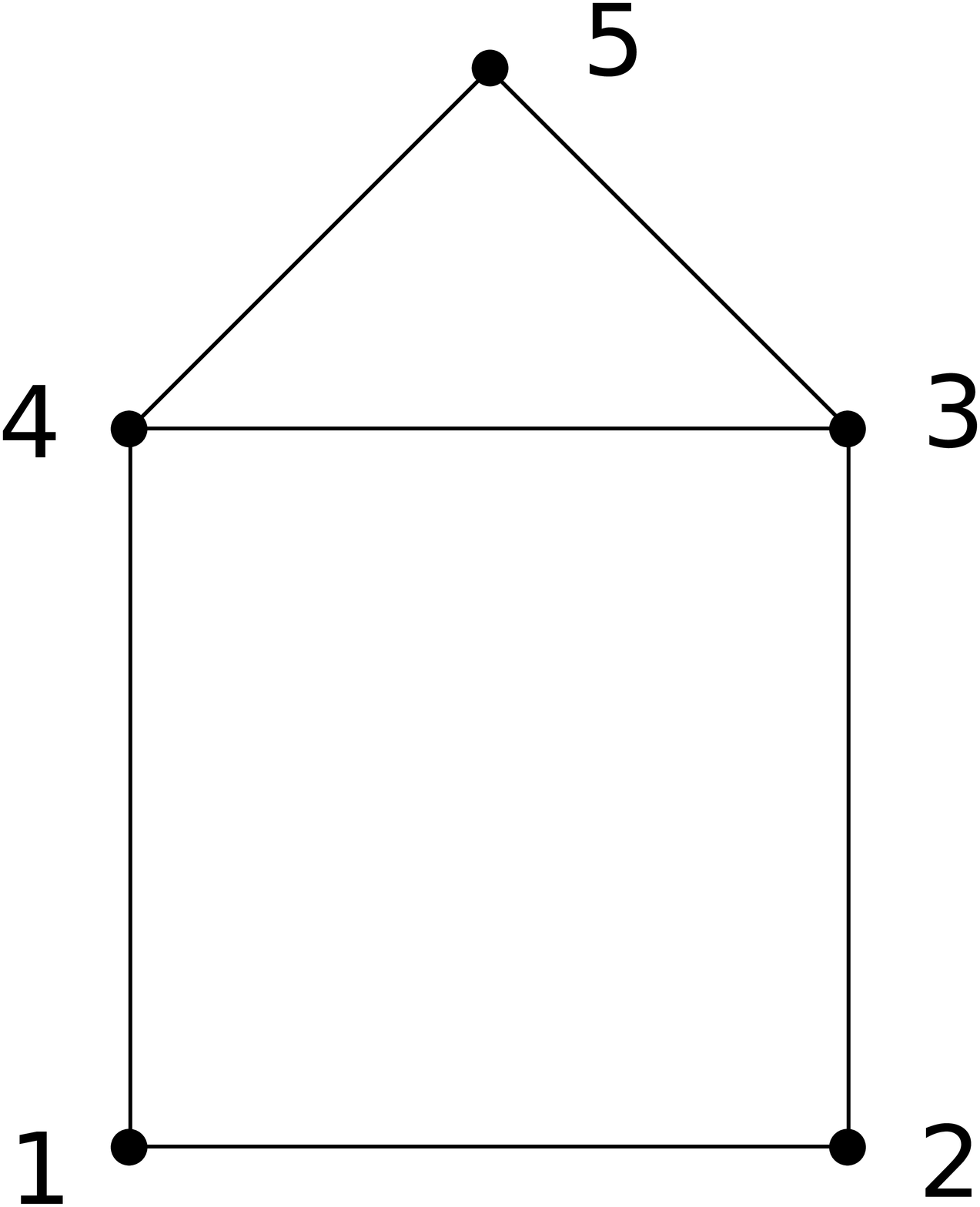}

 \end{center}
  \caption{$C_4 \# C_3$}
  \label{fig:two}
 \end{minipage}
 \begin{minipage}{0.4\hsize}
 \begin{center}
\includegraphics[trim = 90 90 90 0 , scale=.09, clip]{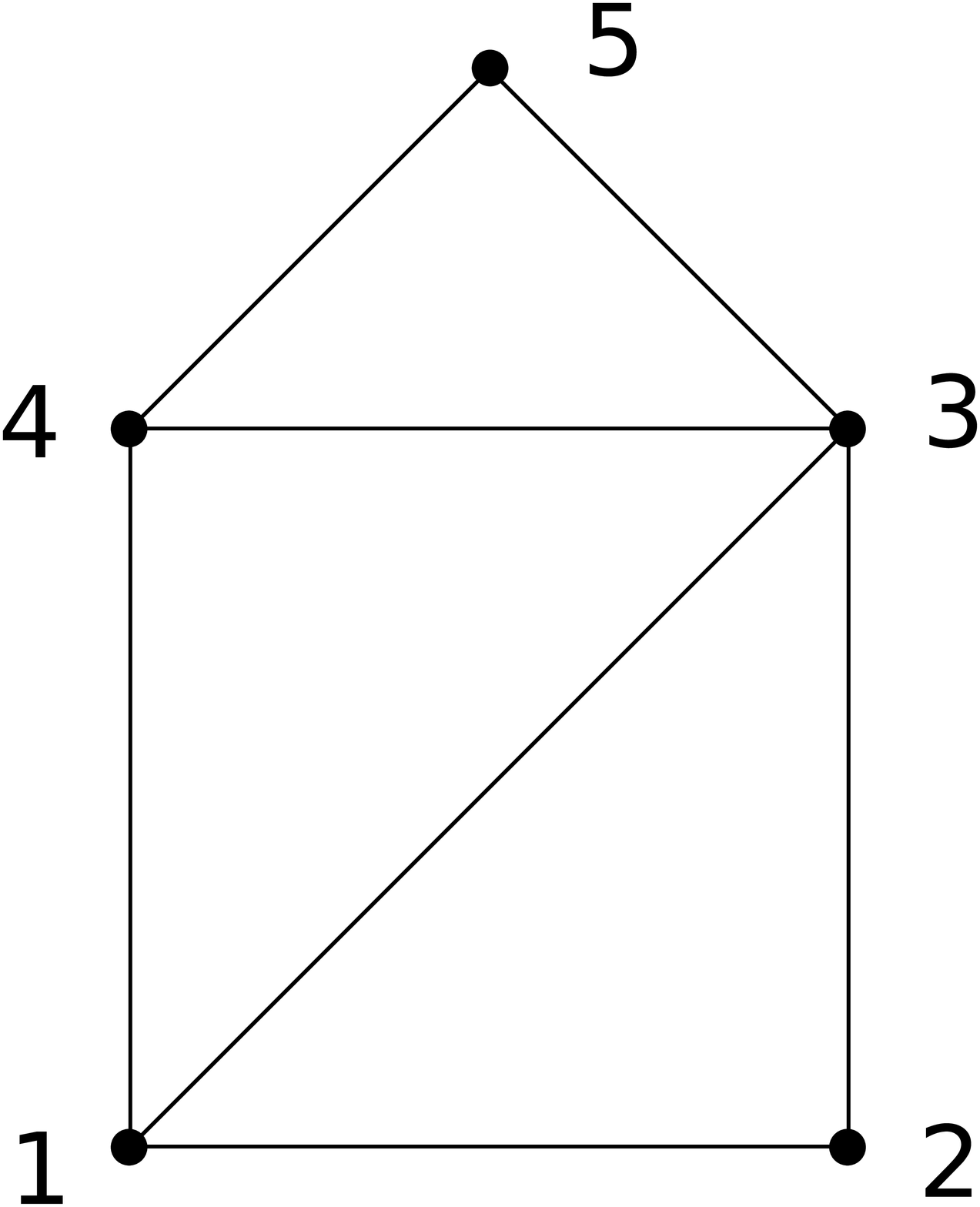}

 \end{center}
  \caption{$(K_4 \setminus e) \# C_3$}
  \label{fig:three}
 \end{minipage}
\end{figure}
\end{center}

\end{example}

\section{A Gr\"obner basis for the cut ideal}
In this section,
we compute a Gr\"obner basis of $I_G$ such that $G$ has no $(K_4,C_5)$-minor.

\begin{lemma}
\lemlab{key}
Let $G$ be a simple 2-connected graph on the vertex set $V(G)$.
Then $G$ has no $(K_4,C_5)$-minor if and only if $G$ is $K_3$, $K_{2,n-2}$ or $K_{1,1,n-2}$ for $n \ge 4$.
\end{lemma}

\begin{proof}
Since $G$ is $2$-connected,
$G$ contains a cycle.
Let $C$ be the longest cycle in $G$.
It follows that $|V(C)| \le 4$ because $G$ has no $C_5$-minor.
If $|V(C)|=3$, then $G=K_3$ since $G$ is $2$-connected.
Suppose that $|V(C)| = 4$.
If $|V(G)| = |V(C)|$, then $G$ is either $K_{2,2}$ or $K_{1,1,2}$.
Next, we assume that $|V(G)| > |V(C)|=4$.
Consider $v \in V(G) \setminus V(C)$.
Let $P$ and $Q$ be two paths each with one end in $v$ and another end in $V(C)$, disjoint except for their common end in $v$ and having no internal vertices in $C$.
Such paths exist since $G$ is 2-connected.
If $|V(P)| > 2$, or $|V(Q)| > 2$, or the ends of $P$ and $Q$ in  $C$ are consecutive in $C$, then $P \cup Q$ together with a subpath of $C$ form a cycle of length longer than $C$.
Hence every vertex $v \notin V(C)$ has exactly two neighbors in $V(C)$, which are not consecutive.
Moreover, if some two vertices $v_1, v_2 \in V(G) \setminus V(C)$ are adjacent to different pairs of vertices in $C$, then a cycle of length six is induced in $G$ by $\{ v_1, v_2 \} \cup V(C)$.
Therefore there exist $u_1 , u_2 \in V(C)$, which are both adjacent to all vertices in $V(G) \setminus \{ u_1 , u_2 \}$.
If two vertices in $V(G) \setminus \{ u_1, u_2 \}$ are adjacent,
then together with $\{ u_1 , u_2 \}$ and any other vertex they induce a cycle in $G$ of length five.
Therefore $G$ is either $K_{2,n-2}$ or $K_{1,1,n-2}$.
It is easy to see that all of $K_3$, $K_{2,n-2}$ and $K_{1,1,n-2}$ have no $(K_4,C_5)$-minor.
\end{proof}

It is already known that the cut ideal $I_{K_{1,n-2}}$ for $n \ge 4$ has a quadratic Gr\"obner basis since $K_{1,n-2}$ is $0$-sums of $K_2$ and $I_{K_2}=\langle 0 \rangle$ \cite[Theorem 2.1]{StuSull}.
In this paper, to prove \thmref{main1}, we compute the reduced Gr\"obner basis of $I_{K_{1,n-2}}$.
Let $<$ be a reverse lexicographic order on $K[q]$ which satisfies $q_{A|B} < q_{C|D}$ with $\min \{ |A|,|B| \} < \min \{ |C| , |D|\}$.

\begin{lemma}
\lemlab{GB1}
Let $G=K_{1,n-2}$ be the complete bipartite graph on the vertex set $V_1 \cup V_2$, where $V_1 = \{ 1 \} $ and $V_2 = \{ 3,\ldots , n \}$ for $n \ge 4$.
Then the reduced Gr\"obner basis of $I_G$ with respect to $<$ consists of
\begin{eqnarray*}
q_{A|B} q_{C|D} - q_{A \cap C | B \cup D} q_{A \cup C | B \cap D}
&
(1 \in A \cap C,A \not\subset C,~C \not\subset A).
\end{eqnarray*}
The initial monomial of each binomial is the first monomial.
\end{lemma}

\begin{proof}
Let $\calg$ be the set of all binomials above.
It is easy to see that $\calg \subset I_G$.
Let ${\rm in}(\calg)=\langle {\rm in}_{<} (g) ~|~ g \in \calg \rangle$.
Let $u$ and $v$ be monomials that do not belong to ${\rm in}(\calg)$:
$$
u =
\prod_{l=1}^{m}
(q_{\{ 1 \} \cup A_l |  B_l})^{p_{l}}, \quad
v =
\prod_{l=1}^{m^{'}}
(q_{\{ 1 \} \cup C_l |  D_l})^{p^{'}_{l}},
$$
where $0<p_l$, $p^{'}_l \in \zz$ for any $l$.
Since neither $u$ nor $v$ is divided by $q_{A|B} q_{C| D}$, 
it follows that
$$
A_1 \subset A_2 \subset \cdots \subset A_{m}, \quad 
C_1 \subset C_2 \subset \cdots \subset C_{m^{'}}.
$$
Let
\begin{eqnarray*}
A_l = A_{l-1} \cup \{ b^{l-1}_1 , \ldots , b^{l-1}_{\beta_{l-1}} \}, & B_k = \bigcup_{i=k}^{m} \{ b^{i}_1 , \ldots , b^{i}_{\beta_{i}} \} \\
C_l = C_{l-1} \cup \{ d^{l-1}_1 , \ldots , d^{l-1}_{\delta_{l-1}} \}, & D_k = \bigcup_{i=k}^{m^{'}} \{ d^{i}_1 , \ldots , d^{i}_{\delta_{i}} \} \\
\end{eqnarray*}
for $k \ge 1$ and $l \ge 2$, where $A_1 = V_2 \setminus B_1$, $C_1 = V_2 \setminus D_1$.
We suppose that $\pi_G (u) = \pi_G(v)$:
$$
\pi_G(u) = s^{p} \prod_{l=1}^{m} (t_{1 b^{l}_{1}} \cdots t_{1 b^{l}_{\beta_l}})^{\sum_{k=1}^{l} p_k}, \quad
\pi_G(v) = s^{p^{'}} \prod_{l=1}^{m^{'}} (t_{1 d^{l}_{1}} \cdots t_{1 d^{l}_{\delta_l}})^{\sum_{k=1}^{l} p^{'}_k}.
$$
Here we set $p = \sum_{l=1}^{m} p_l$ and $p^{'} = \sum_{l=1}^{m^{'}} p^{'}_l$.
Assume that $A_1 \ne C_1$.
Then there exists $a \in A_1$ such that $a \notin C_1$.
Hence, for some $l_1 \in [m^{'}]$, $a \in \{ d^{l_1}_{1} , \ldots , d^{l_1}_{\delta_{l_1}} \}$.
However, for any $l \in [m]$, $a \notin \{ b^{l}_{1} ,\ldots , b^{l}_{\beta_l} \}$.
This contradicts that $\pi_G (u) = \pi_G (v)$.
Thus $A_1 = C_1$ and $p_1 = p^{'}_1$.
By performing this operation repeatedly,
it follows that $A_l = C_l$, $B_l = D_l$ and $p_l = p^{'}_{l}$ for any $l$.
Since $u=v$,
$\calg$ is a Gr\"obner basis of $I_G$.
It is trivial that $\calg$ is reduced.
\end{proof}


\begin{theorem}
\thmlab{main1}
Let $G=K_{2,n-2}$ be the complete bipartite graph on the vertex set $V_1 \cup V_2$, where $V_1 = \{ 1,2 \}$ and $V_2 = \{ 3,\ldots , n\}$ for $n \ge 4$.
Then a Gr\"obner basis of $I_G$ consists of
\begin{eqnarray}
q_{A|B} q_{E|F} - q_{\emptyset | [n]} q_{\{ 1,2 \} | \{ 3,\ldots , n \}} &
(1 \in A , 2 \in B),  \\
q_{A|B} q_{C|D} - q_{A \cap C | B \cup D} q_{A \cup C | B \cap D} &
(1 \in A \cap C, 2 \in B \cap D , A \not\subset C , C \not\subset A), \\
q_{A|B} q_{C|D} - q_{A \cap C | B \cup D} q_{A \cup C | B \cap D} &
(1,2 \in A \cap C, A \not\subset C , C \not\subset A),
\end{eqnarray}
where $E=(B \cup \{ 1 \}) \setminus \{ 2 \}$ and $F = (A \cup \{ 2 \}) \setminus \{ 1 \}$.
The initial monomial of each binomials is the first binomial.
\end{theorem}

\begin{proof}
Let $\calg$ be the set of all binomials above.
It is easy to see that $\calg \subset I_G$.
Let $u$ and $v$ be monomials which do not belong to ${\rm in}(\calg)$:
\begin{eqnarray*}
u &=&
\prod_{l=1}^{m_1} (q_{ \{ 1 \} \cup A_l | \{ 2 \} \cup B_l})^{p_l}
\prod_{l=1}^{m_2} (q_{ \{ 1,2 \} \cup C_l | D_l})^{r_l}, \\
v &=&
\prod_{l=1}^{m^{'}_1} (q_{ \{ 1 \} \cup A^{'}_l |  \{ 2 \} \cup B^{'}_l})^{p^{'}_l}
\prod_{l=1}^{m^{'}_2} (q_{ \{ 1,2 \} \cup  C^{'}_l | D^{'}_l})^{r^{'}_l},
\end{eqnarray*}
where $0 < p_l , r_l,p^{'}_l , r^{'}_l \in \zz$ for any $l$.
Since neither $u$ nor $v$ is divided by initial monomials of (ii) and (iii),
it follows that
\begin{eqnarray*}
A_1 \subset \cdots \subset A_{m_1}, \quad
C_1 \subset \cdots \subset C_{m_2} ,\\
A^{'}_1 \subset \cdots \subset A^{'}_{m^{'}_1}, \quad 
C^{'}_1 \subset \cdots \subset C^{'}_{m^{'}_2}.
\end{eqnarray*}
Suppose that $\pi_G (u) = \pi_G (v)$:
\begin{eqnarray*}
\pi_G (u) &=&
\prod_{l=1}^{m_1} (u_{ \{ 1 \} \cup A_l | \{ 2 \} \cup B_l})^{p_l}
\prod_{l=1}^{m_2} (u_{ \{ 1,2 \} \cup C_l | D_l})^{r_l}, \\
\pi_G (v) &=&
\prod_{l=1}^{m^{'}_1} (u_{ \{ 1 \} \cup A^{'}_l |  \{ 2 \} \cup B^{'}_l})^{p^{'}_l}
\prod_{l=1}^{m^{'}_2} (u_{ \{ 1,2 \} \cup  C^{'}_l | D^{'}_l})^{r^{'}_l}.
\end{eqnarray*}
Let $Y$ be the matrix consisting of the first $n-2$ rows of $X_{K_{1,n-2}}$.
Then $X_{G}$ is the following matrix:
$$
\begin{pmatrix}
Y & Y \\
Y & {\bf 1}_{n-2,2^{n-2}} -Y \\
{\bf 1} & {\bf 1}
\end{pmatrix},
$$
where ${\bf 1}_{n-2,2^{n-2}}$ is the $(n-2) \times 2^{n-2}$ matrix such that each entry is all ones.
Note that
\begin{eqnarray*}
\begin{pmatrix}
Y \\
Y
\end{pmatrix}
&=&
\begin{pmatrix}
\delta_{P_1|Q_1} (K_{2,n-2}) \cdots \delta_{P_{2^{n-2}}|Q_{2^{n-2}}} (K_{2,n-2})
\end{pmatrix} \\
\begin{pmatrix}
Y \\
{\bf 1}_{n-2,2^{n-2}} -Y
\end{pmatrix}
&=&
\begin{pmatrix}
\delta_{R_1|S_1} (K_{2,n-2}) \cdots \delta_{R_{2^{n-2}}|S_{2^{n-2}}} (K_{2,n-2})
\end{pmatrix}
,
\end{eqnarray*}
where $1,2 \in P_l$, $1 \in R_l$ and $2 \in S_l$ for $1 \le l \le 2^{n-2}$.
By elementary row operations of $X_G$,
we have
$$
X^{'}_{G} =
\begin{pmatrix}
2Y -{\bf 1}_{n-2,2^{n-2}} & O \\
O & 2Y - {\bf 1}_{n-2,2^{n-2}} \\
{\bf 1} & {\bf 1}
\end{pmatrix}.
$$
Each column vector of $2Y-{\bf 1}_{n-2,2^{n-2}}$ is the form ${}^t (\varepsilon_1,\ldots,\varepsilon_{n-2})$, where $\varepsilon_i \in \{ 1,-1\}$ for $1 \le i \le n-2$.
Let $I_{X^{'}_{G}}$ denote the toric ideal of $X^{'}_{G}$ (see \cite{Sturmfels}).
Then $u-v \in I_G$ if and only if $u-v \in I_{X^{'}_{G}}$.
Let $\bfa_{P|Q}$ denote the column vector of $2Y-{\bf 1}_{n-2,2^{n-2}}$ in $X^{'}_{G}$ corresponding to the column vector $\delta_{P|Q}(G)$ of $X_G$.
Then
$$
\sum_{l=1}^{m_1}
p_l
\begin{pmatrix}
{\bf 0} \\
\bfa_{\{1 \} \cup A_l | \{ 2 \} \cup B_l } \\
1
\end{pmatrix}
+
\sum_{l=1}^{m_2}
r_l
\begin{pmatrix}
\bfa_{\{1,2 \} \cup C_l | D_l } \\
{\bf 0} \\
1
\end{pmatrix}
=
\sum_{l=1}^{m^{'}_1}
p^{'}_l
\begin{pmatrix}
{\bf 0} \\
\bfa_{\{1 \} \cup A^{'}_l | \{ 2 \} \cup B^{'}_l } \\
1
\end{pmatrix}
+
\sum_{l=1}^{m^{'}_2}
r^{'}_l
\begin{pmatrix}
\bfa_{\{1,2 \} \cup C^{'}_l |D^{'}_l } \\
{\bf 0} \\
1
\end{pmatrix}
.
$$
In particular,
$$
\sum_{l=1}^{m_1}
p_l
\bfa_{\{1 \} \cup A_l | \{ 2 \} \cup B_l } 
=
\sum_{l=1}^{m^{'}_1}
p^{'}_l
\bfa_{\{1 \} \cup A^{'}_l | \{ 2 \} \cup B^{'}_l },
\quad
\sum_{l=1}^{m_2}
r_l
\bfa_{\{1,2 \} \cup C_l | D_l }
=
\sum_{l=1}^{m^{'}_2}
r^{'}_l
\bfa_{\{1,2 \} \cup C^{'}_l |D^{'}_l }
$$
hold.
Let $p=\sum_{l=1}^{m_1} p_l$, $r=\sum_{l=1}^{m_2} r_l$, $p^{'}=\sum_{l=1}^{m^{'}_1} p^{'}_l$ and $r^{'} = \sum_{l=1}^{m^{'}_2} r^{'}_l$.
Since neither $u$ nor $v$ is divided by initial monomials of (i), it follows that either $A_1 \ne \emptyset$ or $A_{m_1} \ne [n] \setminus \{ 1,2 \}$ (resp. $A^{'}_1 \ne \emptyset$ or $A^{'}_{m^{'}_2} \ne [n] \setminus \{ 1,2 \}$).
If $A_1 \ne \emptyset$, then there exists $i \in [n] \setminus \{ 1,2 \}$ such that $i \in A_l$ for any $l \in [m_1]$.
If $A_{m_1} \ne [n] \setminus \{ 1,2\}$, that is, $B_{m_1} \ne \emptyset$, then there exists $i \in [n] \setminus \{ 1,2 \}$ such that $i \in B_{m_1}$, and $i \notin A_l$ for any $l \in [m_1]$.
Thus either $p$ or $-p$ appears in the entry of $\sum_{l=1}^{m_1} p_l \bfa_{\{ 1 \} \cup A_l | \{ 2 \} \cup B_l}$.
Similarly, either $p^{'}$ or $-p^{'}$ appears in the entry of $\sum_{l=1}^{m^{'}_1} p^{'}_l \bfa_{\{ 1 \} \cup A^{'}_l | \{ 2 \} \cup B^{'}_l}$.
Therefore $p=p^{'}$.
Hence
$$
\prod_{l=1}^{m_1} (u_{ \{ 1 \} \cup A_l | \{ 2 \} \cup B_l})^{p_l}
=\prod_{l=1}^{m^{'}_1} (u_{ \{ 1 \} \cup A^{'}_l |  \{ 2 \} \cup B^{'}_l})^{p^{'}_l},
\quad
\prod_{l=1}^{m_2} (u_{ \{ 1,2 \} \cup C_l | D_l})^{r_l}=
\prod_{l=1}^{m^{'}_2} (u_{ \{ 1,2 \} \cup  C^{'}_l | D^{'}_l})^{r^{'}_l}
$$
hold.
Thus
\begin{eqnarray*}
\prod_{l=1}^{m_1} (q_{ \{ 1 \} \cup A_l | \{ 2 \} \cup B_l})^{p_l}
-\prod_{l=1}^{m^{'}_1} (q_{ \{ 1 \} \cup A^{'}_l |  \{ 2 \} \cup B^{'}_l})^{p^{'}_l} \in I_{Z_1},
\\
\prod_{l=1}^{m_2} (q_{ \{ 1,2 \} \cup C_l | D_l})^{r_l}-
\prod_{l=1}^{m^{'}_2} (q_{ \{ 1,2 \} \cup  C^{'}_l | D^{'}_l})^{r^{'}_l}
\in I_{Z_2},
\end{eqnarray*}
where $Z_1$ (resp. $Z_2$) is the matrix consisting of the first (resp. last) $2^{n-2}$ columns of $X^{'}_{G}$.
Here $I_{Z_1}$ and $I_{Z_2}$ are toric ideals of $Z_1$ and $Z_2$.
By elementary row operations of $Z_1$ (resp. $Z_2$),
we have
$$
\prod_{l=1}^{m_1} (q_{ \{ 1 \} \cup A_l |  B_l})^{p_l}
-\prod_{l=1}^{m^{'}_1} (q_{ \{ 1 \} \cup A^{'}_l |  B^{'}_l})^{p^{'}_l},
\quad
\prod_{l=1}^{m_2} (q_{ \{ 1 \} \cup C_l | D_l})^{r_l}-
\prod_{l=1}^{m^{'}_2} (q_{ \{ 1 \} \cup  C^{'}_l | D^{'}_l})^{r^{'}_l}
\in I_{K_{1,n-2}}.
$$
By \lemref{GB1}, $u=v$ holds.
Therefore $\calg$ is a Gr\"obner basis of $I_G$.
\end{proof}

\begin{corollary}
\corlab{cor1}
If $G$ has no $(K_4,C_5)$-minor,
then $I_G$ has a quadratic Gr\"obner basis.
\end{corollary}

\begin{proof}
If $G$ is not 2-connected,
then there exist 2-connected components $G_1, \ldots , G_s$ of $G$ such that $G$ is $0$-sums of $G_1,\ldots,G_s$.
By \cite{StuSull} and \lemref{key},
it is enough to show that, $I_{K_2}$, $I_{K_3}$, $I_{K_{2,n-2}}$ and $I_{K_{1,1,n-2}}$ have a quadratic Gr\"obner basis.
It is trivial that $I_{K_2}$ and $I_{K_3}$ have a quadratic Gr\"obner basis because $I_{K_2}=\langle 0 \rangle$ and $I_{K_3} = \langle 0 \rangle$.
Since $K_{1,1,n-2}$ is obtained by $1$-sums of $C_3$,
$I_{K_{1,1,n-2}}$ has a quadratic Gr\"obner basis.
Therefore, by \thmref{main1}, $I_G$ has a quadratic Gr\"obner basis.
\end{proof}

\section{Strongly Koszul toric rings of cut ideals}

In this section, we characterize the class of graphs whose toric rings associated to cut ideals are strongly Koszul.


\begin{proposition}
\proplab{str}
Let $G_1 = K_{1,1,n-2}$ and $G_2=K_{2,n-2}$ for $n \ge 4$.
Then $R_{G_1}$ and $R_{G_2}$ are strongly Koszul.
\end{proposition}

\begin{proof}
By elementary row operations of $X_{G_1}$,
we have
$$
X_{G_1} =
\begin{pmatrix}
{\bf 0} & {\bf 1} \\
Y & Y \\
Y & {\bf 1}_{n-2,2^{n-2}} - Y \\
{\bf 1} & {\bf 1} \\
\end{pmatrix}
\rightarrow
\begin{pmatrix}
{\bf 0} & {\bf 1} \\
Y & Y \\
Y & - Y \\
{\bf 1} & {\bf 1} \\
\end{pmatrix}
\rightarrow
\begin{pmatrix}
{\bf 0} & {\bf 1} \\
Y & Y \\
Y & O \\
{\bf 1} & {\bf 1} \\
\end{pmatrix}
\rightarrow
\begin{pmatrix}
{\bf 0} & {\bf 1} \\
O & Y \\
Y & O \\
{\bf 1} & {\bf 0} \\
\end{pmatrix}.
$$
Hence $R_{G_1} \cong R_{K_{1,n-2}} \otimes_K R_{K_{1,n-2}}$.
Since $R_{K_{1,n-2}}$ is Segre products of $R_{K_2}$, $R_{G_1}$ is strongly Koszul.
Next, by the symmetry of $X_{G^{'}}$ in the proof of \thmref{main1}, it is enough to consider the following two cases:
\begin{itemize}
\item[(1)] $( u_{\emptyset | [n]}) \cap (u_{ \{ 1 \} | \{ 2,\ldots , n \} })$,
\item[(2)]
$( u_{\emptyset | [n]}) \cap (u_{\{ 1,2 \} \cup A | B })$.
\end{itemize}
\noindent
Since $q_{\emptyset|[n]}$ is the smallest variable and $q_{\{ 1 \}|\{ 2,\ldots , n\}}$ is the second smallest variable with respect to the reverse lexicographic order $<$,
by \cite{MatsudaOhsugi} and \thmref{main1}, $( u_{\emptyset | [n]}) \cap (u_{ \{ 1 \} | \{ 2,\ldots , n \} })$ is generated in degree 2.
Assume that $(u_{\emptyset |[n]}) \cap (u_{ \{ 1,2 \} \cup A|B})$ is not generated in degree 2.
Then there exists a monomial $u_{E_1|F_1} \cdots u_{E_s|F_s}$ belonging to a minimal generating set of $(u_{\emptyset |[n]}) \cap (u_{\{ 1,2 \} \cup A|B})$ such that $s \ge 3$.
Since $u_{E_1|F_1} \cdots u_{E_s|F_s}$ is in $(u_{\emptyset |[n]}) \cap (u_{ \{ 1,2 \} \cup A|B})$,
it follows that
$$
q_{\{ 1,2 \} \cup A|B}
\prod_{l=1}^{\alpha}
q_{\{ 1,2 \} \cup A_l|B_l}
\prod_{l=1}^{\beta}
q_{\{ 1 \} \cup C_l| \{ 2 \} \cup D_l}
-
q_{\emptyset | [n]}
\prod_{l=1}^{\gamma}
q_{\{ 1,2 \} \cup P_l|Q_l}
\prod_{l=1}^{\delta}
q_{\{ 1 \} \cup R_l|\{ 2 \} \cup S_l}
\in I_{G_2}.
$$
If one of the monomials appearing in the above binomial is divided by initial monomials of (i) in \thmref{main1},
then $u_{E_1|F_1}\cdots u_{E_s | F_s}$ is divided by $u_{\emptyset | [n]} u_{\{1,2\}|\{ 3,\ldots , n\}}$.
This contradicts that $u_{E_1|F_1}\cdots u_{E_s | F_s}$ belongs to a minimal generating set of $(u_{\emptyset | [n]}) \cap (u_{\{ 1,2 \} \cup A|B})$ since, for any $u_{A|B}$ and $u_{C|D}$ with $u_{A|B} \ne u_{C|D}$, $u_{\emptyset | [n]}u_{\{ 1,2 \}|\{ 3,\ldots , n\}}$ belongs to a minimal generating set of $(u_{A|B}) \cap (u_{C|D})$.
If one of $\prod_{l=1}^{\beta} q_{\{ 1 \} \cup C_l| \{ 2 \} \cup D_l}$ and $\prod_{l=1}^{\delta} q_{\{ 1 \} \cup R_l| \{ 2 \} \cup S_l}$ is divided by initial monomials of (ii) in \thmref{main1},
the monomial is reduced to the monomial which is not divided by initial monomials of (ii) with respect to $\calg$, where $\calg$ is a Gr\"obner basis of $I_{G_2}$.
Thus we may assume that
$$
C_1 \subset \cdots \subset C_{\beta}, \quad
R_1 \subset \cdots \subset R_{\delta}.
$$
Similar to what did in the proof of \thmref{main1}, we have
\begin{eqnarray*}
u_{\{ 1,2 \} \cup A|B}
\prod_{l=1}^{\alpha}
u_{\{ 1,2 \} \cup A_l|B_l}
&=&
u_{\emptyset | [n]}
\prod_{l=1}^{\gamma}
u_{\{ 1,2 \} \cup P_l|Q_l}
, \\
\prod_{l=1}^{\beta}
u_{\{ 1 \} \cup C_l|\{ 2 \} \cup D_l}
&=&
\prod_{l=1}^{\delta}
u_{\{ 1 \} \cup R_l|\{ 2 \} \cup S_l}.
\end{eqnarray*}
It follows that $\alpha = \gamma$, $\beta=\delta$, $C_l = R_l$, $D_l=S_l$ for any $l$, and
$$
q_{\{ 1 \} \cup A|B}
\prod_{l=1}^{\alpha}
q_{\{ 1 \} \cup A_l|B_l}
-
q_{\emptyset | [n] \setminus \{ 2 \}}
\prod_{l=1}^{\alpha}
q_{\{ 1 \} \cup P_l|Q_l}
\in I_{K_{1,n-2}}.
$$
Hence the ideal $(u_{\{ 1 \} \cup A|B}) \cap (u_{\emptyset | [n] \setminus \{ 2 \}})$ of $R_{K_{1,n-2}}$ is not generated in degree $2$.
However this contradicts that $R_{K_{1,n-2}}$ is strongly Koszul.
Therefore $R_{G_2}$ is strongly Koszul.
\end{proof}

\begin{lemma}
\lemlab{lem2}
Let $G$ be a finite simple 2-connected graph with no $K_4$-minor.
If $G$ has $C_5$-minor,
then by only contracting edges of $G$, we obtain one of $C_5$, the 1-sum of $C_4$ and $C_3$, and the 1-sum of $K_4 \setminus e$ and $C_3$.
\end{lemma}

\begin{proof}
Let $G$ be a graph with $C_5$-minor and $C$ be a longest cycle in $G$.
It follows that $|V(C)| \ge 5$.
Then, by contracting edges of $G$, we obtain a graph $G^{'}$ of five vertices such that $C_5$ is a subgraph of $G^{'}$.
Assume that $G^{'} \ne C_5$.
Then there exist $u,v \in V(C_5)$ with $uv \notin E(C_5)$ such that $uv \in E(G^{'})$.
Since $G$ has no $K_4$-minor,
there do not exist $\alpha, \beta \in V(C_5) \setminus \{ u,v \}$ such that $\alpha \beta \in E(G^{'}) \setminus E(C_5)$.
Therefore we obtain one of the $1$-sum of $C_4$ and $C_3$, and the $1$-sum of $K_4 \setminus e$ and $C_3$.
\end{proof}

\begin{theorem}
Let $G$ be a finite simple connected graph.
Then $R_{G}$ is strongly Koszul if and only if $G$ has no ($K_4$, $C_5$)-minor.
\end{theorem}

\begin{proof}
Let $G$ be a graph with no $(K_4,C_5)$-minor.
If $G$ is not $2$-connected,
then there exist $2$-connected components $G_1, \ldots, G_s$ of $G$ such that $G$ is $0$-sums of $G_1,\ldots, G_s$.
By \lemref{key},
it is enough to show that $R_{K_2}$, $R_{K_3}$, $R_{K_{2,n-2}}$ and $R_{K_{1,1,n-2}}$ are strongly Koszul.
It is clear that $R_{K_2}$ and $R_{K_3}$ are strongly Koszul.
By \propref{str},
$R_{K_{2,n-2}}$ and $R_{K_{1,1,n-2}}$ are strongly Koszul.
Next, we suppose that $G$ has $K_4$-minor.
Then the cut ideal $I_G$ is not generated by quadratic binomials \cite{Engstrom}.
In particular, $R_G$ is not strongly Koszul.
Assume that $G$ has no $K_4$-minor.
If $G$ has $C_5$-minor,
then, by \lemref{lem2}, we obtain one of $C_5$, $C_4 \# C_3$ and $(K_4 \setminus e) \# C_3$ by contracting edges of $G$.
By \exref{exam}, neither $R_{C_4 \# C_3}$ nor $R_{(K_4 \setminus e) \# C_3}$ is strongly Koszul.
By \cite[Theorem 1.3]{StuSull}, since $R_{C_5}$ is not compressed, $R_{C_5}$ is not strongly Koszul \cite[Theorem 2.1]{MatsudaOhsugi}.
Therefore, by \propref{contraction}, $R_G$ is not strongly Koszul.
\end{proof}
By using above results,
we have
\begin{corollary}
The set of graphs $G$ such that $R_G$ is strongly Koszul is minor closed.
\end{corollary}

\begin{corollary}
If $R_G$ is strongly Koszul,
then $I_G$ has a quadratic Gr\"obner basis.
\end{corollary}

\section*{Acknowledgement}
The author would like to thank Hidefumi Ohsugi for useful comments and suggestions.

\end{document}